\documentclass[11pt]{amsart}
\usepackage{amssymb}
\usepackage{amsaddr}
\usepackage[margin=1in]{geometry}
\usepackage[colorlinks]{hyperref}
\usepackage{graphicx}

\newtheorem{theorem}{Theorem}[section]
\newtheorem{lem}[theorem]{Lemma}
\newtheorem{prop}[theorem]{Proposition}
\newtheorem{cor}[theorem]{Corollary}

\theoremstyle{definition}
\newtheorem{definition}[theorem]{Definition}
\newtheorem{example}[theorem]{Example}

\theoremstyle{remark}
\newtheorem{remark}[theorem]{Remark}

\numberwithin{equation}{section}

\begin{document}

\newcommand{\spacing}[1]{\renewcommand{\baselinestretch}{#1}\large\normalsize}
\spacing{1.14}
\title[f-left-invariant Riemannian metrics on Lie groups]{f-left-invariant Riemannian metrics on Lie groups}

\author {Hamid Reza Salimi Moghaddam}

\address{Department of Pure Mathematics, Faculty of  Mathematics and Statistics, University of Isfahan,
Isfahan, 81746-73441-Iran.\\E-Mails: salimi.moghaddam@gmail.com and hr.salimi@sci.ui.ac.ir\\
Scopus Author ID: 26534920800 \\ ORCID Id:0000-0001-6112-4259}


\date{\today}

\begin{abstract}
With a f-left-invariant Riemannian metric on a Lie group $G$, we mean a Riemannian metric which is conformally equivalent to a left-invariant Riemannian metric, with the conformal factor $f$. In this article, we study the geometry of such metrics and give a necessary and sufficient condition for an f-left-invariant Riemannian metric to be a Ricci soliton. Using this result, for any expansion constant $\lambda$, we obtain a flat gradient Ricci soliton on some two and three-dimensional non-abelian Lie groups. We give an example of a non-flat steady gradient Ricci soliton and construct some examples of non-flat shrinking, steady, and expanding non-gradient Ricci solitons on the non-abelian Lie group $\Bbb{R}\rtimes\Bbb{R}^{+}$. Finally, we study f-left-invariant Riemannian metrics on the Heisenberg group.\\
\textbf{Keywords:} conformally equivalent metric, left-invariant Riemannian metric, Lie group, Ricci soliton. \\
\textbf{AMS 2020 Mathematics Subject Classification:} 53C30, 53C21, 22E60.
\end{abstract}

\maketitle

\section{\textbf{Introduction}}\label{Introduction}
A Riemannian manifold $(M,g)$ is named a Ricci soliton if there exists a smooth vector field $X$ and a constant $\lambda\in\Bbb{R}$ (which is called expansion constant) such that:
\begin{equation}\label{Ricci soliton equation}
    \mathcal{L}_Xg=2(\lambda g-\mbox{Ric}_g),
\end{equation}
where $\mathcal{L}_Xg$ denotes the Lie derivative of $g$ with respect to $X$ and $\mbox{Ric}_g$ is the Ricci tensor of $g$.
The group $\textsf{Diff}(M)$ of diffeomorphisms of $M$ and the group $\Bbb{R}^+$ of positive real numbers act on the space of Riemannian metrics on $M$ as follows:
\begin{eqnarray*}
  \psi. g &=& \psi^\ast g \ \ \ \forall \psi\in \textsf{Diff}(M), \\
  c. g &=& cg \ \ \ \forall c\in\Bbb{R}^\ast.
\end{eqnarray*}
It can be demonstrated that if $(M,g)$ is a Ricci soliton with a vector field $X$ and an expansion constant $\lambda$, and $\psi\in \textsf{Diff}(M)$ and $c\in\Bbb{R}^+$, then $(M,\psi^\ast g)$ is also a Ricci soliton with the vector field $\psi^\ast X$ and the same expansion constant $\lambda$. Additionally, $(M,cg)$ is a Ricci soliton with the vector field $\frac{1}{c}X$ and the expansion constant $\frac{\lambda}{c}$. Ricci solitons are regarded as the self-similar solutions of the Ricci flow equation since they are fixed points of the Ricci flow equation in the space of Riemannian metrics modulo the above action (see \cite{Chow}).
The notion of Ricci solitons can be considered as a generalization of Einstein manifolds. Furthermore, the study of Ricci solitons, particularly algebraic Ricci solitons, has played an essential role in the study of Alekseevskii conjecture, which is finally proved by  B\"ohm and Lafuente in \cite{Bohm-Lafuente} (for more details see \cite{Jablonski2}).\\
Ricci solitons are classified as shrinking (when $\lambda>0$), steady (when $\lambda=0$), or expanding (when $\lambda<0$). If there exists a smooth real-valued function $\Phi$ such that $X=\mbox{grad}\Phi$, then the Ricci soliton is known as a gradient Ricci soliton, where the function $\Phi$ is named as a potential function of the Ricci soliton (see \cite{Chow}).\\
Since Hamilton's seminal paper, \cite{Hamilton}, on Ricci flow and Perelman's proof of the Poincar\'{e} conjecture, \cite{Perelman}, various mathematicians have focused on finding examples and classifying Ricci solitons (see \cite{Bernstein-Mettler}, \cite{Bohm-Lafuente}, \cite{Cao-Chen}, \cite{Catino-Mastrolia-Monticelli-Rigoli}, \cite{Chow}, \cite{Djiadeu-Boucetta-Wouafo}, \cite{Jablonski1}, \cite{Jablonski2} and \cite{Salimi}). For instance, Bernstein and Mettler classified non-compact two-dimensional gradient Ricci solitons, \cite{Bernstein-Mettler}, while Baird described a class of three-dimensional Ricci solitons and provided some examples (see \cite{Baird}).\\
An intriguing research area within this field involves studying Ricci solitons that are conformally equivalent to specific Riemannian metrics or have a Lie group or homogeneous space as the base space. For example, Cao and Chen, \cite{Cao-Chen}, examined locally conformally flat gradient steady Ricci solitons, while Klepikov, \cite{Klepikov}, investigated conformally flat algebraic Ricci solitons on Lie groups and characterized their Ricci operator.\\
This article focuses on Ricci solitons that are conformally equivalent to left-invariant Riemannian metrics, which are referred to as f-left-invariant metrics, with $f$ as the conformal factor. In a sense, f-left-invariant Riemannian metrics lie between left-invariant Riemannian metrics and general Riemannian metrics. Section 2 presents a comprehensive theory of f-left-invariant Riemannian metrics and establishes a necessary and sufficient condition for them to be Ricci solitons. Section 3 explores f-left-invariant Ricci solitons on two-dimensional Lie groups and provides examples of abelian flat shrinking, steady, and expanding gradient Ricci solitons, as well as abelian flat steady non-gradient Ricci solitons. In the non-abelian case, examples of non-flat shrinking, expanding, and steady non-gradient Ricci solitons are given, along with an example of a non-flat steady gradient Ricci soliton. Moreover, examples of non-abelian flat shrinking, expanding, and steady gradient Ricci solitons are provided. In Section 4, the focus shifts to the three-dimensional Lie groups$\Bbb{R}^2\rtimes\Bbb{R}^{+}$ and the Heisenberg group. It is shown that the non-abelian Lie group $\Bbb{R}^2\rtimes\Bbb{R}^{+}$ admits flat shrinking, expanding, and steady gradient and non-gradient Ricci soliton structures. Finally, the Heisenberg group is studied, and an example of a left-invariant non-gradient expanding Ricci soliton is constructed using the methodology presented in this article. This approach can be applied to construct examples of Ricci solitons in higher-dimensional Lie groups or be generalized to homogeneous spaces.


\section{\textbf{Riemannian metrics which are conformally equivalent to the left-invariant metrics}}\label{conformal metrics}
In this section, our focus revolves around the examination of Riemannian metrics on Lie groups that are conformally equivalent to left-invariant Riemannian metrics. These metrics are referred to as f-left-invariant metrics, where $f$ represents the conformal factor. We begin with a definition:
\begin{definition}
Consider a smooth, real, positive function $f$ on a Lie group $G$, with $f(e)=1$, where $e$ denotes the unit element of $G$. A Riemannian metric $\tilde{g}$ is classified as f-left-invariant (f-right-invariant) if it is conformally equivalent to a left-invariant (right-invariant) metric $g$, with $f$ serving as the conformal factor. If a Riemannian metric possesses both f-left-invariant and f-right-invariant characteristics, it is referred to as an f-bi-invariant metric. In this case, $f$ must be a class function.
\end{definition}

\begin{remark}\label{generalized cases}
It is worth noting that the above definition can be generalized to semi-Riemannian metrics and Finsler metrics similarly.
\end{remark}

\begin{remark}\label{existence of metric}
Let $f$ be a smooth, positive real function on a Lie group $G$ such that $f(e)=1$. It can be readily observed that $G$ admits a left-invariant Riemannian metric $g$. Consequently, $G$ also possesses an f-left-invariant Riemannian metric $\tilde{g}$, defined as follows for any $a\in G$:
\begin{equation}\label{eq2}
    \tilde{g}_a(X_a,Y_a):=f(a)\tilde{g}_e(L_{a^{-1}\ast}X_a,L_{a^{-1}\ast}Y_a).
\end{equation}
In the case of f-right-invariant we will have a similar result.
\end{remark}
\begin{lem}\label{f-left and f-right condition}
Suppose that $f$ is a smooth real positive function on a Lie group $G$ such that $f(e)=1$. A Riemannian metric $\tilde{g}$ on $G$ is $f$-left-invariant if and only if, for any $a,b\in G$,
\begin{equation}\label{f-left}
\tilde{g}_{ba}(L_{b\ast}X_a,L_{b\ast}Y_a)=\frac{f(ba)}{f(a)}\tilde{g}_a(X_a,Y_a).
\end{equation}
Similarly, the Riemannian metric $\tilde{g}$ is $f$-right-invariant if and only if
\begin{equation}\label{f-right}
\tilde{g}_{ba}(R_{b\ast}X_a,R_{b\ast}Y_a)=\frac{f(ab)}{f(a)}\tilde{g}_a(X_a,Y_a),
\end{equation}
for all $a,b\in G$.
\end{lem}
\begin{proof}
It suffices to prove the case for f-left-invariant metrics, as the other case follows similarly. Assume that $\tilde{g}$ is a Riemannian metric on $G$ that satisfies equation (3). It can be observed that $\tilde{g}$ is conformally equivalent to the left-invariant Riemannian metric $g$, induced by the inner product $\tilde{g}_e$, with the conformal factor $f$. Conversely, if $\tilde{g}$ is a Riemannian metric on $G$ that is conformally equivalent to a left-invariant Riemannian metric $g$, then for any $a, b\in G$ and any two left-invariant vector fields $X$ and $Y$, we have:
\begin{equation}
    \tilde{g}_{ba}(L_{b\ast}X_a,L_{b\ast}Y_a)=f(ba)g_{ba}(X_{ba},Y_{ba})=f(ba)g_a(X_{a},Y_{a})=\frac{f(ba)}{f(a)}\tilde{g}_a(X_a,Y_a).
\end{equation}
\end{proof}
\begin{prop}\label{equivalence}
Let $G$ denote a connected Lie group equipped with a Riemannian metric $\tilde{g}$ that is $f$-left-invariant. Then, the following statements are equivalent:
\begin{enumerate}
  \item $\tilde{g}$ is $f$-right-invariant, hence $f$-bi-invariant.
  \item $\tilde{g}$ is $Ad(G)$-invariant.
  \item $f(a)\tilde{g}_{a^{-1}}(\zeta_\ast X_a,\zeta_\ast Y_a)=f(a^{-1})\tilde{g}_a(X_a,Y_a)$, for all $a\in G$, where $\zeta$ is the inversion map.
  \item $\tilde{g}(X,[Y,Z])=\tilde{g}([X,Y],Z)$, for all $X,Y,Z\in\frak{g}$.
\end{enumerate}
\end{prop}
\begin{proof}
First, we will demonstrate the equivalence between statements (1) and (2). Let us assume that condition (1) holds. According to Lemma \ref{f-left and f-right condition}, we can deduce the following:
\begin{eqnarray}
   \tilde{g}_e(Ad_aX_e,Ad_aY_e)&=&\tilde{g}_e(L_{a\ast}R_{a^{-1}\ast}X_e,L_{a\ast}R_{a^{-1}\ast}Y_e) \nonumber  \\
   &=&\tilde{g}_{a^{-1}}(R_{a^{-1}\ast}X_e,R_{a^{-1}\ast}Y_e)\frac{f(aa^{-1})}{f(a^{-1})}  \\
   &=& \tilde{g}_e(X_e,Y_e)\frac{f(ea^{-1})}{f(e)}\frac{f(e)}{f(a^{-1})}=\tilde{g}_e(X_e,Y_e).\nonumber
\end{eqnarray}
Conversely, we will establish that condition (2) implies condition (1). The following calculations demonstrate this:
\begin{eqnarray}
   \tilde{g}_{a^{-1}}(R_{a^{-1}\ast}X_e,R_{a^{-1}\ast}Y_e)&=&\tilde{g}_e( L_{a\ast}R_{a^{-1}\ast}X_e,L_{a\ast}R_{a^{-1}\ast}Y_e)\frac{f(a^{-1})}{f(aa^{-1})} \nonumber \\
   &=&\tilde{g}_e(Ad_aX_e,Ad_aY_e) f(a^{-1})  \\
   &=& \tilde{g}_e(X_e,Y_e) f(a^{-1}).\nonumber
\end{eqnarray}
Next, we will demonstrate the equivalence of conditions (1) and (3). Suppose that condition (1) holds. Then, we have:
\begin{eqnarray}
   \tilde{g}_{a^{-1}}(\zeta_\ast X_a,\zeta_\ast Y_a)&=&\tilde{g}_{a^{-1}}(R_{a^{-1}\ast}\zeta_\ast L_{a^{-1}\ast}X_a,R_{a^{-1}\ast}\zeta_\ast L_{a^{-1}\ast}Y_a) \nonumber  \\
                                           &=&\tilde{g}_{e}(\zeta_\ast L_{a^{-1}\ast}X_a,\zeta_\ast L_{a^{-1}\ast}Y_a)f(a^{-1}) \nonumber  \\
                                           &=&\tilde{g}_{e}(-L_{a^{-1}\ast}X_a,-L_{a^{-1}\ast}Y_a)f(a^{-1})  \\
                                           &=&\tilde{g}_{a}(X_a,Y_a)\frac{f(a^{-1})}{f(a)}.\nonumber
\end{eqnarray}
Conversely, let us assume that condition (3) holds. Then, we have:
\begin{eqnarray}
   \tilde{g}_a(R_{a\ast}X_e,R_{a\ast}Y_e)&=&\tilde{g}_a(\zeta_\ast L_{a^{-1}\ast}\zeta_\ast X_e,\zeta_\ast L_{a^{-1}\ast}\zeta_\ast Y_e)\nonumber  \\
                              &=&\tilde{g}_{a^{-1}}(L_{a^{-1}\ast}\zeta_\ast X_e,L_{a^{-1}\ast}\zeta_\ast Y_e)\frac{f(a)}{f(a^{-1})}  \\
                              &=&\tilde{g}_e(\zeta_\ast X_e,\zeta_\ast Y_e)\frac{f(a^{-1}e)}{f(e)}\frac{f(a)}{f(a^{-1})}\nonumber\\
                              &=&\tilde{g}_e(X_e,Y_e)f(a).\nonumber
\end{eqnarray}
The equivalence of conditions $(2)$ and $(4)$ is similar to the invariant Riemannian metric case, so we omit it (see lemma 3 page 302 of \cite{Oneill}).
\end{proof}

\begin{cor}
Let $f$ be a smooth positive real function on a compact Lie group $G$, such that $f(e)=1$. Then $G$ admits a $f$-bi-invariant Riemannian metric.
\end{cor}

Suppose that $g$ is an arbitrary Riemannian metric on an $n$-dimensional manifold $M$. We denote the curvature tensor and the Ricci tensor by
\begin{equation}\label{Curvature Tensor}
    R(X,Y)Z=\nabla_X\nabla_YZ-\nabla_Y\nabla_XZ-\nabla_{[X,Y]}Z,
\end{equation}
and
\begin{equation}\label{Ricci Tensor}
    Ric(Y,Z)=tr(X\longrightarrow R(X,Y)Z),
\end{equation}
where $\nabla$ is the Levi-Civita connection of $g$.\\
For a smooth function $\psi$ on the Riemannian manifold $(M,g)$, the gradient, the Hessian form and the Hessian $(1,1)$-tensor denoted by $\nabla\psi$, $\nabla^2\psi$ and $H_\psi(X):=\nabla_X(\nabla\psi)$, respectively.
So we have $\nabla^2\psi(X,Y)=g(\nabla_X(\nabla\psi),Y)$. The Laplacian of $\psi$ is defied by $\Delta\psi:=tr(H_\psi)$.
\begin{lem}\label{Lemma curvature}
Let $\tilde{g}=fg$ be a $f$-left-invariant Riemannian metric on an $n$-dimensional Lie group $G$, where $f=e^{-2\phi}$, for a smooth function $\phi:G\to\Bbb{R}$ with $\phi(e)=0$. Then for the Levi-Civita connection $\tilde{\nabla}$, the curvature tensor $\tilde{R}$ and the Ricci tensor $\tilde{Ric}$ of the Riemannian metric $\tilde{g}$, for all left-invariant vector fields $X, Y$ and $Z$, we have:
\begin{enumerate}
  \item \begin{eqnarray*}
            \tilde{\nabla}_XY=\frac{1}{2}\Big{(}[X,Y]-ad_X^\ast Y-ad_Y^\ast X\Big{)}-\Big{(}(X\phi)Y+(Y\phi)X-g(X,Y)\nabla\phi\Big{)},
        \end{eqnarray*}
  \item \begin{eqnarray*}
          \tilde{R}(X,Y)Z &=& \frac{1}{4}\Big{(}ad_Zad_XY-ad_X^\ast(ad_YZ-ad_Y^\ast Z-ad_Z^\ast Y)\\
          && +ad^\ast_Y(ad_XZ-ad_X^\ast Z-ad_Z^\ast X)-ad_X(ad_Y^\ast Z+ad_Z^\ast Y) \\
          && +ad_Y(ad^\ast_XZ+ad^\ast_ZX)+(ad^\ast_{ad^\ast_YZ}+ad^\ast_{ad^\ast_ZY}-ad^\ast_{ad_YZ})X\\
          && -(ad^\ast_{ad^\ast_XZ}+ad^\ast_{ad^\ast_ZX}-ad^\ast_{ad_XZ})Y\Big{)}\\
          && +\frac{1}{2}(ad^\ast_{ad_XY}Z+ad^\ast_Z ad_XY)-(g(X,Z)H_\phi Y-g(Y,Z)H_\phi X)\\
          && +\Big{(}\nabla^2\phi(Y,Z)+(Y\phi)(Z\phi)-g(Y,Z)\|\nabla\phi\|^2 \Big{)}X\\
          && -\Big{(}\nabla^2\phi(X,Z)+(X\phi)(Z\phi)-g(X,Z)\|\nabla\phi\|^2 \Big{)}Y\\
          && +\Big{(}(X\phi)g(Y,Z)-(Y\phi)g(X,Z)\Big{)}\nabla\phi,
        \end{eqnarray*}
  \item \begin{eqnarray*}
          \tilde{Ric}(X,Y) &=& -\frac{1}{2}\Big{(}tr(ad_X\circ ad_Y)+tr(ad_X\circ ad^\ast_Y)+g(ad_HX,Y)+g(ad_HY,X)\Big{)}\\
          && -\frac{1}{4}tr(J_X\circ J_Y)+\Big{(}\Delta\phi-(n-2)\|\nabla\phi\|^2\Big{)}g(X,Y)+\frac{n-2}{\psi}\nabla^2\psi(X,Y),
        \end{eqnarray*}
\end{enumerate}
where $\psi=e^\phi$, $H$ is the mean curvature vector on the Lie algebra $\frak{g}$ of $G$, defined by $g(H,X)=tr(ad_X)$, and $J_X$ denotes the endomorphism defined by $J_X=ad^\ast_YX$.
We mention that all quantities of the right hand sides of the above equations are computed with respect to the left-invariant Riemannian metric $g$.
\end{lem}
\begin{proof}
For the left-invariant Riemannian metric $g$ we have:
\begin{equation*}
    \nabla_XY=\frac{1}{2}\Big{(}[X,Y]-ad^\ast_XY-ad^\ast_YX\Big{)}
\end{equation*}
and
\begin{eqnarray*}
  R(X,Y)Z &=& \frac{1}{4}\Big{(}ad_Zad_XY-ad_X^\ast(ad_YZ-ad_Y^\ast Z-ad_Z^\ast Y)\\
          && +ad^\ast_Y(ad_XZ-ad_X^\ast Z-ad_Z^\ast X)-ad_X(ad_Y^\ast Z+ad_Z^\ast Y) \\
          && +ad_Y(ad^\ast_XZ+ad^\ast_ZX)+(ad^\ast_{ad^\ast_YZ}+ad^\ast_{ad^\ast_ZY}-ad^\ast_{ad_YZ})X\\
          && -(ad^\ast_{ad^\ast_XZ}+ad^\ast_{ad^\ast_ZX}-ad^\ast_{ad_XZ})Y\Big{)}\\
          && +\frac{1}{2}(ad^\ast_{ad_XY}Z+ad^\ast_Z ad_XY).
\end{eqnarray*}
On the other hand, based on the formula $(3)$ of \cite{Djiadeu-Boucetta-Wouafo}, for the left-invariant Riemannian metric $g$ we have:
\begin{equation*}
   Ric(X,Y)=-\frac{1}{2}\Big{(}tr(ad_X\circ ad_Y)+tr(ad_X\circ ad^\ast_Y)+g(ad_HX,Y)+g(ad_HY,X)\Big{)}
           -\frac{1}{4}tr(J_X\circ J_Y).
\end{equation*}
Now Lemma 1 of \cite{Kuhnel} completes the proof.
\end{proof}

The following proposition is a direct consequence of Lemma \ref{Lemma curvature} and formula \eqref{Ricci soliton equation}.
\begin{prop}
Suppose that $\tilde{g}=fg$ is a $f$-left-invariant Riemannian metric on an $n$-dimensional Lie group $G$ where $f=e^{-2\phi}$, $\phi(e)=0$, and $\psi=e^\phi$. Then the Riemannian manifold $(G,\tilde{g})$ is a Ricci soliton if and only if there exist a vector field $X$ and a constant $\lambda$ such that for any $Y, Z \in \frak{g}$, we have:
\begin{eqnarray}\label{Ricci soliton conformal}
        && -\frac{1}{2}\Big{(}tr(ad_X\circ ad_Y)+tr(ad_X\circ ad^\ast_Y)+g(ad_HX,Y)+g(ad_HY,X)\Big{)}\\
        && -\frac{1}{4}tr(J_X\circ J_Y)+\Big{(}\Delta\phi-(n-2)\|\nabla\phi\|^2\Big{)}g(X,Y)+\frac{n-2}{\psi}\nabla^2\psi(X,Y)\nonumber \\
        && -\lambda fg(Y,Z)+\frac{1}{2}(Xf)g(Y,Z)-\frac{1}{2}f(\mathcal{L}_Xg)(Y,Z)=0.    \nonumber
\end{eqnarray}
\end{prop}

We see that the equation \eqref{Ricci soliton conformal} is not very simple for computation. So in the following we give an equivalent formula based on the structural constants of the Lie algebra $\frak{g}$ of $G$. At the first we compute the sectional and Ricci curvatures of f-left-invariant Riemannian metrics using structural constants.

In this article, we use the notation $\{E_1,\cdots,E_n\}$ for a set of left-invariant vector fields on a Lie group $G$ which is an orthogonal basis at any point of $G$ and is an orthonormal basis at the unit element $e$, with respect to a $f$-left-invariant Riemannian metric $\tilde{g}$.

\begin{prop}\label{sectional theorem}
Let $G$ be a Lie group equipped with a $f$-left-invariant Riemannian metric $\tilde{g}$. Suppose that $\alpha_{ijk}$ are structure constants defined by $[E_i,E_j]=\sum_{k=1}^n\alpha_{ijk}E_k$. Then the sectional curvature $\tilde{K}(E_p,E_q)$ is given by the following formula:
\begin{eqnarray}\label{sectional curvature general}
  \tilde{K}(E_p,E_q) &=& \frac{1}{4f^3}\Big{(}2(f_{p}^2+f_{q}^2)-2f(f_{pp}+f_{qq})\nonumber\\
             && \hspace{1cm} +\sum_{h=1}^n \big{(} (2\delta_{qh}f_q-f_h+2f\alpha_{hqq})(f_h+2f\alpha_{php})\nonumber\\
             && \hspace{2cm} -(\delta_{qh}f_p+\delta_{hp}f_q+f(\alpha_{hpq}+\alpha_{pqh}-\alpha_{qhp}))\\
             && \hspace{2cm} \times (\delta_{hp}f_q-\delta_{qh}f_p+f(\alpha_{pqh}+\alpha_{qhp}-\alpha_{hpq}))\nonumber\\
             && \hspace{2cm} -2f\alpha_{pqh}(\delta_{ph}f_q-\delta_{hq}f_p+f(\alpha_{phq}+\alpha_{hqp}-\alpha_{qph}))\big{)}\Big{)},\nonumber
\end{eqnarray}
where $f_i:=E_if$ and $f_{ij}:=E_jE_if$.
\end{prop}
\begin{proof}
The relation $[E_i,E_j]=\sum_{k=1}^n\alpha_{ijk}E_k$ shows that
\begin{equation}\label{alpha}
    \alpha_{ijk}=\frac{1}{f(a)}\tilde{g}_a([E_i,E_j],E_k).
\end{equation}
Therefore we have
\begin{eqnarray}
  2\tilde{g}(\nabla_{E_i}E_j,E_k) &=& E_i\tilde{g}(E_j,E_k)\tilde{g}+E_j\tilde{g}(E_i,E_k)-E_k\tilde{g}(E_i,E_j) \nonumber\\
                         && -\tilde{g}(E_i,[E_j,E_k])+\tilde{g}(E_j,[E_k,E_i])+\tilde{g}(E_k,[E_i,E_j]) \\
                         &=& \delta_{jk}f_i+\delta_{ik}f_j-\delta_{ij}f_k+f(-\alpha_{jki}+\alpha_{kij}+\alpha_{ijk}),\nonumber
\end{eqnarray}
and so,
\begin{equation}\label{Levi-Civita connection}
    \nabla_{E_i}E_j=\frac{1}{2f}\sum_{k=1}^n\Big{(}\delta_{jk}f_i+\delta_{ik}f_j-\delta_{ij}f_k+f(\alpha_{ijk}+\alpha_{kij}-\alpha_{jki})\Big{)}E_k.
\end{equation}
Now, for the curvature tensor we have:
\begin{eqnarray}\label{Curvature Tensor}
  \tilde{R}(E_i,E_j)E_k &=&\tilde{\nabla}_{E_i}\tilde{\nabla}_{E_j}E_k-\tilde{\nabla}_{E_j}\tilde{\nabla}_{E_i}E_k-\tilde{\nabla}_{[E_i,E_j]}E_k\nonumber\\
                &=& \frac{1}{4f^2}\sum_{l=1}^n\Big{(} 2f(\delta_{kl}(f_{ji}-f_{ij})+\delta_{lj}f_{ki}-\delta_{jk}f_{li}-\delta_{li}f_{kj}+\delta_{ik}f_{lj})\nonumber\\
                && \hspace{1.5cm} + 2(\delta_{li}f_kf_j-\delta_{ik}f_lf_j-\delta_{lj}f_kf_i+\delta_{jk}f_lf_i)\nonumber\\
                && \hspace{1.5cm} + \sum_{h=1}^n \big{(} (\delta_{kh}f_j+\delta_{hj}f_k-\delta_{jk}f_h+f(\alpha_{hjk}+\alpha_{jkh}-\alpha_{khj})) \nonumber\\
                && \hspace{2.7cm} \times (\delta_{hl}f_i+\delta_{li}f_h-\delta_{ih}f_l+f(\alpha_{lih}+\alpha_{ihl}-\alpha_{hli})) \\
                && \hspace{2.7cm} - (\delta_{kh}f_i+\delta_{hi}f_k-\delta_{ik}f_h+f(\alpha_{hik}+\alpha_{ikh}-\alpha_{khi}))\nonumber\\
                && \hspace{2.7cm} \times (\delta_{hl}f_j+\delta_{lj}f_h-\delta_{jh}f_l+f(\alpha_{ljh}+\alpha_{jhl}-\alpha_{hlj}))\nonumber\\
                && \hspace{2.7cm} -2f\alpha_{ijh}(\delta_{kl}f_h+\delta_{lh}f_k-\delta_{hk}f_l+f(\alpha_{lhk}+\alpha_{hkl}-\alpha_{klh})) \big{)}\Big{)} E_l.\nonumber
\end{eqnarray}
On the other hand,
\begin{equation}\label{eq3}
    \tilde{g}(E_p,E_p)\tilde{g}(E_q,E_q)-\tilde{g}(E_p,E_q)^2=f^2.
\end{equation}
Finally, the proof is completed by using the formula for sectional curvature.
\end{proof}

\begin{remark}
When considering left-invariant Riemannian metrics on Lie groups, the formula for sectional curvature provided in Theorem \ref{sectional theorem} reduces to Milnor's formula as stated
in \cite{Mi}. This simplification occurs when we set $f$ to be the constant function $f\equiv1$. Furthermore, if the Riemannian metric $\tilde{g}$ is $f$-bi-invariant, then the structural
constants $\alpha_{ijk}$ are skew in the last two indices for any $i$. Therefore, we obtain a simpler formula for sectional curvature in this case.
\end{remark}
In the case where $G$ is an abelian Lie group, we can derive a straightforward formula for sectional curvature.
\begin{cor}\label{abelian case}
Let $G$ be a commutative Lie group equipped with a $f$-left-invariant Riemannian metric $\tilde{g}$. Then for the sectional curvature, we have
\begin{equation}\label{Sectional Commutative}
    \tilde{K}(E_{p},E_{q})=\frac{1}{4f^3}\Big{(}3(f_p^2+f_q^2)-2f(f_{pp}+f_{qq})-\sum_{l=1}^nf_l^2\Big{)}.
\end{equation}
\end{cor}
Now we give the Ricci curvature of a f-left-invariant Riemannian metric.

\begin{prop}\label{Ricci theorem}
Under the assumptions of Proposition \ref{sectional theorem}, the Ricci curvature tensor $\tilde{\mbox{Ric}}$ is given by the following formula:
\begin{eqnarray}\label{Ricci curvature general}
  \tilde{\mbox{Ric}}(E_p,E_q) &=& \frac{1}{4f^2}\sum_{i=1}^n\Big{(}2f\big{(}\delta_{qi}(f_{pi}-f_{ip})+\delta_{ip}f_{qi}-\delta_{pq}f_{ii}+\delta_{iq}f_{ip}-f_{qp}\big{)}\nonumber\\
                                &&\hspace*{1cm} +2(f_qf_p-\delta_{iq}f_if_p-\delta_{ip}f_qf_i+\delta_{pq}f_i^2)\nonumber\\
                                &&\hspace*{1cm} +\sum_{h=1}^n \big{(} (2f\alpha_{ihi}+f_h)(\delta_{qh}f_p+\delta_{hp}f_q-\delta_{pq}f_h+f(\alpha_{hpq}+\alpha_{pqh}-\alpha_{qhp})) \\
                                &&\hspace*{2cm} -(\delta_{qh}f_i+\delta_{hi}f_q-\delta_{iq}f_h+f(\alpha_{hiq}+\alpha_{iqh}-\alpha_{qhi}))\nonumber\\
                                &&\hspace*{2cm} \times(\delta_{hi}f_p+\delta_{ip}f_h-\delta_{ph}f_i+f(\alpha_{iph}+\alpha_{phi}-\alpha_{hip})) \nonumber\\
                                &&\hspace*{2cm}-2f\alpha_{iph}(\delta_{qi}f_h+\delta_{ih}f_q-\delta_{hq}f_i+f(\alpha_{ihq}+\alpha_{hqi}-\alpha_{qih}))\big{)}\Big{)}.\nonumber
\end{eqnarray}
\end{prop}
\begin{proof}
Let us assume that $e_i:=\frac{E_i}{\sqrt{f}}$, where $i=1,\cdots, n$. It is important to note that, in general, the vector fields $e_i$ are not left-invariant. However, we can observe that the set $\{e_1,\cdots, e_n\}$ forms an orthonormal basis at every point of $G$ with respect to the $f$-left-invariant metric. Now, considering the equation
\begin{equation}
   \tilde{\mbox{Ric}}(E_p,E_q)=f\tilde{\mbox{Ric}}(e_p,e_q)=f\sum_{j=1}^n\tilde{g}( \tilde{R}(e_j,e_p)e_q,e_j)=\frac{1}{f}\sum_{j=1}^n\tilde{g}( \tilde{R}(E_j,E_p)E_q,E_j),
\end{equation}
along with the formula \eqref{Curvature Tensor}, we can conclude the proof.
\end{proof}

\begin{remark}
The previous proposition demonstrates that for left-invariant Riemannian metrics, we have the following expression for the Ricci curvature:
\begin{eqnarray}\label{Ricci curvature left-invariant Riemannian metrics}
  \mbox{Ric}(E_p,E_q) &=& \frac{1}{4}\sum_{i=1}^n\sum_{h=1}^n\Big{(}2\alpha_{ihi}(\alpha_{hpq}+\alpha_{pqh}-\alpha_{qhp})-(\alpha_{hiq}+\alpha_{iqh}-\alpha_{qhi})(\alpha_{iph}+\alpha_{phi}-\alpha_{hip})\nonumber\\
  && \hspace*{2cm} -2\alpha_{iph}(\alpha_{ihq}+\alpha_{hqi}-\alpha_{qih})\Big{)}.
\end{eqnarray}
\end{remark}

Using the proposition mentioned above, we can derive the following corollary for Ricci soliton $f$-left-invariant Riemannian metrics.

\begin{cor}\label{Ricci soliton condition for f-left-invariant}
Let $G$ be a connected Lie group equipped with a $f$-left-invariant Riemannian metric $\tilde{g}$. Suppose that $X=\sum_{i=1}^n\theta^iE_i$ is an arbitrary vector field on $G$, which is not necessarily left-invariant so the coefficients $\theta^i$ are smooth functions on $G$. Then, $(G,\tilde{g})$ is a Ricci soliton, with expansion constant $\lambda$ and the vector field $X$, if and only if, for any $p,q=1,\cdots, n$:
\begin{equation}\label{necessary and sufficient condition f-invariant}
             \delta_{pq}(Xf)+f(\theta_p^q+\theta_q^p)-\sum_{i=1}^n\theta^if(\alpha_{ipq}+\alpha_{iqp})-2(\delta_{pq}\lambda f-\tilde{\mbox{Ric}}(E_p,E_q))=0
\end{equation}
\end{cor}
\begin{proof}
We can easily observe that for $\mathcal{L}_X\tilde{g}$, we have:
\begin{equation}\label{Lie derivative}
    \mathcal{L}_X\tilde{g}(E_p,E_q)=\delta_{pq}(Xf)+f(\theta_p^q+\theta_q^p)-\sum_{i=1}^n\theta^if(\alpha_{ipq}+\alpha_{iqp}).
\end{equation}
Now, combining the above equation with equations \eqref{Ricci soliton equation} and \eqref{Ricci curvature general}, we can conclude the proof.
\end{proof}

\section{\textbf{f-left-invariant Ricci solitons on two-dimensional Lie groups}}\label{two-dimensional Lie groups}
In this section, we present the necessary and sufficient conditions for $f$-left-invariant Riemannian metrics to be Ricci solitons on simply connected two-dimensional Lie groups.
We then reconstruct Hamilton's cigar soliton using $f$-left-invariant Riemannian metrics and provide some examples of such Ricci solitons.\\
First, we note that a simply connected two-dimensional Lie group, up to automorphisms of Lie groups, can be either the abelian Lie group $\mathbb{R}^2$ or the non-abelian solvable Lie group $\mathbb{R}\rtimes\mathbb{R}^{+}$.
\subsection{\textbf{Lie group $G=\Bbb{R}^2$}}\label{Lie group G=R^2}
Up to isometry, the only left-invariant Riemannian metric on $G$ is the metric $g$ such that the set $\left\{\frac{\partial}{\partial x}, \frac{\partial}{\partial y}\right\}$ forms an orthonormal basis at every point. Let $\tilde{g}$ be an arbitrary $f$-left-invariant Riemannian metric on $G$, which is conformally equivalent to $g$ with the conformal factor $f$. We define $E_1:=\frac{\partial}{\partial x}$ and $E_2:=\frac{\partial}{\partial y}$, and observe that $\{E_1, E_2\}$ forms an orthogonal basis at any point and is orthonormal at $e=(0,0)$ with respect to $\tilde{g}$. In this case, the Levi-Civita connection of $\tilde{g}$ is given by:
\begin{eqnarray*}\label{Levi-Civita R^2}
 \begin{tabular}{|c|c|c|}
  \hline
   $\tilde{\nabla}$ & $E_1$ & $E_2$  \\
  \hline
  $E_1$ & $\frac{1}{2f}(f_1E_1-f_2E_2)$ & $\frac{1}{2f}(f_2E_1+f_1E_2)$ \\
  \hline
  $E_2$ & $\frac{1}{2f}(f_2E_1+f_1E_2)$ & $\frac{1}{2f}(-f_1E_1+f_2E_2)$ \\
  \hline
 \end{tabular}
\end{eqnarray*}

Suppose that $X=\alpha\frac{\partial}{\partial x}+\beta\frac{\partial}{\partial y}$ is an arbitrary vector field on $G$, where $\alpha$ and $\beta$ are smooth real functions on $G$. Then easily we can see the equation \eqref{Ricci soliton equation} reduces to the following system of three equations,
\begin{equation}\label{Ricci Soliton eq for R^2}
 \left\{
  \begin{array}{l}
    \alpha f_x+\beta f_y+2f\alpha_x=2(\lambda-\kappa)f \\
    \alpha f_x+\beta f_y+2f\beta_y=2(\lambda-\kappa)f \\
    \beta_x=-\alpha_y,
  \end{array}
\right.
\end{equation}
where $\kappa=\frac{f_x^2+f_y^2}{2f^3}-\frac{f_{xx}+f_{yy}}{2f^2}$ is the Gaussian curvature of $G$. Additionally, we can deduce that $X=\text{grad}\Phi$ if and only if the following equations hold:
\begin{equation}\label{gradient eq for R^2}
 \left\{
  \begin{array}{l}
    \Phi_x=f\alpha \\
    \Phi_y=f\beta.
  \end{array}
\right.
\end{equation}
Now, let us consider the case where $f(x,y)=\frac{1}{1+x^2+y^2}$ and $\lambda=0$. By substituting these values, we obtain $X=-2x\frac{\partial}{\partial x}-2y\frac{\partial}{\partial y}$ and $\Phi(x,y)=-\ln(1+x^2+y^2)$, which corresponds to the well-known Hamilton's cigar soliton.

In the following example for any expansion constant $\lambda\in\Bbb{R}$ we give a flat gradient Ricci soliton.

\begin{example}
Suppose that $f(x,y)=e^{x+y}$. For an arbitrary real number $\lambda$ let $X=\lambda\frac{\partial}{\partial x}+\lambda\frac{\partial}{\partial y}$. Now the relations \eqref{Ricci Soliton eq for R^2} and \eqref{gradient eq for R^2} show that $(\Bbb{R}^2, \tilde{g})$ is a flat shrinking (if $\lambda>0$), steady (if $\lambda=0$) or expanding (if $\lambda<0$) gradient Ricci soliton with the potential function $\Phi(x,y)=\lambda e^{x+y}$.\\
At the same time if we consider $\lambda=0$ and $X=\frac{\partial}{\partial x}-\frac{\partial}{\partial y}$ then $(\Bbb{R}^2, \tilde{g})$ is a flat steady Ricci soliton which is not gradient.
\end{example}

Maybe someone asks if we can characterize all Ricci solitons as f-left-invariant? The following simple example gives a negative answer to this question. In the following, we present a gradient steady Ricci soliton that is not conformally equivalent to left-invariant Riemannian metrics.

\begin{example}
Let $G$ be the commutative Lie group $\Bbb{R}^2$. Suppose that $g$ is the Riemannian metric on $\Bbb{R}^2$ such that the set $\{X(p),Y(p)\}$ is an orthonormal basis for $T_p\Bbb{R}^2$, where in the standard coordinates of $\Bbb{R}^2$, $X=\frac{\partial}{\partial x}$ and $Y=y\frac{\partial}{\partial x}+\frac{\partial}{\partial y}$. In fact, in the standard coordinates of $\Bbb{R}^2$ we have
\begin{equation}\label{non-conformal Riemannian metric}
    g=\left(
  \begin{array}{cc}
    1 & -y \\
    -y & 1+y^2 \\
  \end{array}
\right).
\end{equation}
We can see that the Riemannian metric $g$ is not conformally equivalent to a left-invariant Riemannian metric on $\Bbb{R}^2$. If the metric $g$ is conformally equivalent to a left-invariant metric $\tilde{g}$ then they must induce the same inner product on the tangent space $T_{(0,0)}\Bbb{R}^2$. The set $\{\frac{\partial}{\partial x}|_{(0,0)}, \frac{\partial}{\partial y}|_{(0,0)}\}$ is an orthonormal set with respect to the inner product induced by $g$. So the only possibility for a left-invariant metric $\tilde{g}$, to be conformally equivalent to $g$ is the standard metric of $\Bbb{R}^2$. For any $p\in\Bbb{R}^2$, the set $\{\frac{\partial}{\partial x}|_p, \frac{\partial}{\partial y}|_p\}$ is an orthogonal set with respect to $\tilde{g}$ but this set is not orthogonal with respect to $g$ unless $p=(0,0)$. So, by considering the Lemma \ref{f-left and f-right condition}, the Riemannian $g$ is not conformally equivalent to $\tilde{g}$.\\
Now we show that $(\Bbb{R}^2,g)$ is a gradient steady Ricci soliton. For the Levi-Civita connection of $(\Bbb{R}^2,g)$ we have
\begin{equation}\label{Levi-Civita connection of g}
    \nabla_XX=\nabla_YY=\nabla_XY=\nabla_YX=0.
\end{equation}
So $\mbox{Ric}_{g}(X,X)=\mbox{Ric}_{g}(Y,Y)=\mbox{Ric}_{g}(X,Y)=0$. Suppose that $W=\theta X+\mu Y=(\theta+\mu y)\frac{\partial}{\partial x}+\mu\frac{\partial}{\partial y}$ is an arbitrary vector field on $\Bbb{R}^2$. Then the Ricci soliton equation $\mathcal{L}_Wg=2(\lambda g-\mbox{Ric}_{g})$ (see \eqref{Ricci soliton equation}), reduces to the following system
\begin{eqnarray*}
  && \theta_x =\lambda, \\
  && y\mu_x + \mu_y=\lambda,\\
  && \mu_x + y\theta_x + \theta_y=0.
\end{eqnarray*}
Easily we see that $\lambda=0$ and any $\theta,\mu\in\Bbb{R}$ satisfy the above equations. Therefore, for any $\theta,\mu\in\Bbb{R}$, $(\Bbb{R}^2,g)$ with $W=(\theta+\mu y)\frac{\partial}{\partial x}+\mu\frac{\partial}{\partial y}$ is a steady Ricci soliton. Also we can see if $\theta=0$ then $(\Bbb{R}^2,g)$ with $W=\mu y\frac{\partial}{\partial x}+\mu\frac{\partial}{\partial y}$ is a steady gradient Ricci soliton with potential function $\Phi(x,y)=\mu y$.
\end{example}

\subsection{\textbf{Lie group $G=\Bbb{R}\rtimes\Bbb{R}^{+}$}}\label{Lie group G=R*R+}
Let us now consider the non-abelian solvable Lie group $\mathbb{R}\rtimes\mathbb{R}^{+}$. We assume that $G$ is equipped with the $f$-left-invariant Riemannian metric $\tilde{g}$, where the set $\{E_1:=y\frac{\partial}{\partial y}, E_2:=y\frac{\partial}{\partial x}\}$ serves as an orthogonal basis at any point and is orthonormal at $e=(0,1)$. Here, we adopt the natural coordinates $(x,y)$ for $G=\mathbb{R}\rtimes\mathbb{R}^{+}$, with $y>0$.

In this particular case, we have $\alpha_{122}=-\alpha_{212}=1$, while all other structural constants are zero.

Let $X=\alpha E_1+\beta E_2$ be an arbitrary vector field on $G$. The Levi-Civita connection of $\tilde{g}$ is given by the following table:

\begin{eqnarray*}\label{Levi-Civita R times R+}
 \begin{tabular}{|c|c|c|}
  \hline
   $\tilde{\nabla}$ & $E_1$ & $E_2$  \\
  \hline
  $E_1$ & $\frac{1}{2f}(f_1E_1-f_2E_2)$ & $\frac{1}{2f}(f_2E_1+f_1E_2)$ \\
  \hline
  $E_2$ & $\frac{1}{2f}(f_2E_1+(f_1-2f)E_2)$ & $\frac{1}{2f}((2f-f_1)E_1+f_2E_2)$ \\
  \hline
 \end{tabular}
\end{eqnarray*}
The equation \eqref{Ricci soliton equation} reveals that the Riemannian manifold $(G,\tilde{g})$, together with the vector field $X$ and expansion constant $\lambda$, constitutes a Ricci soliton if and only if the following system of equations holds:
\begin{equation}\label{Ricci Soliton eq for R rtimes R^+}
 \left\{
  \begin{array}{l}
    2f(\lambda-y\alpha_y)+\frac{y^2}{f}(f_{xx}+f_{yy})-\frac{y^2}{f^2}(f_x^2+f_y^2)-y\alpha f_y-y\beta f_x+2=0,\\
    2f(\lambda-\alpha+y\beta_x)+\frac{y^2}{f}(f_{xx}+f_{yy})-\frac{y^2}{f^2}(f_x^2+f_y^2)-y\alpha f_y-y\beta f_x+2=0,\\
    \beta+y\beta_y+y\alpha_x=0.
  \end{array}
\right.
\end{equation}
In the coordinates $(x,y)$ the Riemannian metric $\tilde{g}$ is of the form $\left(
                                                                               \begin{array}{cc}
                                                                                 \frac{f}{y^2} & 0 \\
                                                                                 0 & \frac{f}{y^2} \\
                                                                               \end{array}
                                                                             \right)$, so we can see $X=\mbox{grad}\Phi$ if and only if
\begin{equation}\label{gradient eq for R rtimes R^+}
 \left\{
  \begin{array}{l}
    \Phi_x=\frac{f\beta}{y} \\
    \Phi_y=\frac{f\alpha}{y}.
  \end{array}
\right.
\end{equation}
Furthermore, by utilizing equation \eqref{sectional curvature general} to determine the Gaussian curvature of this manifold, we obtain:
\begin{equation}\label{Gaussian curvature for R rtimes R^+}
  \kappa=\frac{y^2(f_x^2+f_y^2)}{2f^3}-\frac{2f+y^2(f_{xx}+f_{yy})}{2f^2}.
\end{equation}
Now using the above results we construct non-flat and flat shrinking, steady and expanding Ricci solitons on the non-abelian Lie group $G=\Bbb{R}\rtimes\Bbb{R}^{+}$.
\begin{example}
Suppose that $f(x,y)=\frac{1}{y}$. Then, easily for the Gaussian curvature we have $\kappa=-\frac{3}{2}y$. Now for any $\lambda\in\Bbb{R}$ let $\alpha=-2\lambda+3y$ and $\beta=\frac{-2\lambda x+C_1}{y^3}$ or equivalently let $X=\frac{-2\lambda x+C_1}{y^2}\frac{\partial}{\partial x}+(3y^2-2\lambda y)\frac{\partial}{\partial y}$, where $C_1\in\Bbb{R}$. Using the systems \eqref{Ricci Soliton eq for R rtimes R^+} and \eqref{gradient eq for R rtimes R^+} we have the following results:
\begin{enumerate}
  \item If $C_1=0$ and $\lambda>0$ then $(G,\tilde{g})$ is a non-flat shrinking non-gradient Ricci soliton.
  \item If $C_1=0$ and $\lambda<0$ then $(G,\tilde{g})$ is a non-flat expanding non-gradient Ricci soliton.
  \item If $C_1\neq 0$ and $\lambda=0$ then $(G,\tilde{g})$ is a non-flat steady non-gradient Ricci soliton.
  \item If $\lambda=C_1=0$ then $(G,\tilde{g})$ is a non-flat steady gradient Ricci soliton with potential function $\Phi(x,y)=3\log(y)+C_2$, where $C_2$ is an arbitrary real number.
\end{enumerate}
\end{example}

\begin{example}
If we put $f(x,y)=y^2$ then easily the Gaussian curvature equals to zero. Now for any $\lambda\in\Bbb{R}$ let $\alpha=\lambda$ and $\beta=\frac{\lambda x+C_1}{y}$ or equivalently let $X=(\lambda x+C_1)\frac{\partial}{\partial x}+(\lambda y)\frac{\partial}{\partial y}$, where $C_1\in\Bbb{R}$. The systems \eqref{Ricci Soliton eq for R rtimes R^+} and \eqref{gradient eq for R rtimes R^+} show that, based on the choice of $\lambda$, the non-abelian Riemannian Lie group $(G=\Bbb{R}\rtimes\Bbb{R}^{+},\tilde{g})$ is a flat shrinking (if $\lambda>0$), steady (if $\lambda=0$) or expanding (if $\lambda<0$) gradient Ricci soliton with the potential function $\Phi(x,y)=\frac{\lambda}{2}(x^2+y^2)+C_1x+C_2$, where $C_1$ and $C_2$ are constant real numbers.
\end{example}

\section{\textbf{f-left-invariant Ricci solitons on $\Bbb{R}^2\rtimes\Bbb{R}^{+}$ and the Heisenberg group}}\label{three-dimensional Lie groups}
In this section, we delve into the investigation of $f$-left-invariant Riemannian metrics on two three-dimensional Lie groups: $\mathbb{R}^2\rtimes\mathbb{R}^{+}$ and the Heisenberg group. We aim to establish the necessary and sufficient conditions for these metrics to qualify as Ricci solitons. Additionally, we provide a collection of examples of $f$-left-invariant Ricci solitons on $\mathbb{R}^2\rtimes\mathbb{R}^{+}$.

\subsection{\textbf{Lie group $G=\Bbb{R}^2\rtimes\Bbb{R}^{+}$}}\label{Lie group G=R^2*R+}
In this subsection we consider the Lie group $G=\Bbb{R}^2\rtimes\Bbb{R}^{+}$ with natural coordinates $(x,y,z)$ such that $z>0$. We consider a $f$-left-invariant Riemannian metric $\tilde{g}$ such that the set $\{E_1:=z\frac{\partial}{\partial z}, E_2:=z\frac{\partial}{\partial x}, E_3:=z\frac{\partial}{\partial y}\}$ is an orthogonal basis at any point and is orthonormal at $e=(0,0,1)$. Easily we can see $\alpha_{122}=\alpha_{133}=1$,  $\alpha_{212}=\alpha_{313}=-1$ and the other structural constants are zero. So for the Levi-Civita connection of $\tilde{g}$ we have:

\begin{eqnarray*}\label{Levi-Civita R2*R+}
\resizebox{\textwidth}{!}{%
 \begin{tabular}{|c|c|c|c|}
  \hline
   $\tilde{\nabla}$ & $E_1$ & $E_2$ & $E_3$  \\
  \hline
  $E_1$ & $\frac{1}{2f}(f_1E_1-f_2E_2-f_3E_3)$ & $\frac{1}{2f}(f_2E_1+f_1E_2)$ & $\frac{1}{2f}(f_3E_1+f_1E_3)$ \\
  \hline
  $E_2$ & $\frac{1}{2f}(f_2E_1+(f_1-2f)E_2)$ & $\frac{1}{2f}((-f_1+2f)E_1+f_2E_2-f_3E_3)$ & $\frac{1}{2f}(f_3E_2+f_2E_3)$ \\
  \hline
  $E_3$ & $\frac{1}{2f}(f_3E_1+(f_1-2f)E_3)$ & $\frac{1}{2f}(f_3E_2+f_2E_3)$ & $\frac{1}{2f}((-f_1+2f)E_1-f_2E_2+f_3E_3)$ \\
  \hline
 \end{tabular}}
\end{eqnarray*}

Suppose that $X=\alpha E_1+\beta E_2+\gamma E_3$ is an arbitrary vector field on $G$, where $\alpha, \beta$ and $\gamma$ are smooth functions on $G$. Then the equations \eqref{necessary and sufficient condition f-invariant} together with \eqref{Ricci curvature general} show that the Riemannian manifold $(G,\tilde{g})$ with the vector field $X$ and expansion constant $\lambda$ is a Ricci soliton if and only if, in the coordiants $(x,y,z)$ the following system holds,

\begin{small}
\begin{equation}\label{Ricci Soliton eq for R^2 rtimes R^+}
 \left\{
  \begin{array}{l}
    z(\alpha f_z+\beta f_x+\gamma f_y)+2zf\alpha_z=2\Big{(}\lambda f+\frac{1}{4f^2}\big{(}8f^2+z^2(4ff_{zz}+2ff_{xx}+2ff_{yy}-4f_z^2-f_x^2-f_y^2)\big{)}\Big{)}\\
    f(\beta+z\beta_z+z\alpha_x)=-\frac{z}{2f^2}(-2zff_{xz}-2ff_x+3zf_xf_z)\\
    f(\gamma+z\beta_z+z\alpha_y)=-\frac{z}{2f^2}(3zf_yf_z-2ff_y-2zff_{yz})\\
    z(\alpha f_z+\beta f_x+\gamma f_y)+2f(z\beta_x-\alpha)=2\Big{(}\lambda f+\frac{1}{4f^2}\big{(}8f^2+z(-4ff_z+2zff_{zz}+4ff_{xx}+2zff_{yy}-zf_z^2\\
    \hspace{11cm} -4zf_x^2-zf_y^2)\big{)}\Big{)}\\
    f(z\gamma_x+z\beta_y)=\frac{z^2}{2f^2}(2ff_{xy}-3f_xf_y)\\
    z(\alpha f_z+\beta f_x+\gamma f_y)+2f(z\gamma_y-\alpha)=2\Big{(}\lambda f+\frac{1}{4f^2}\big{(}8f^2-zff_z+z^2(2ff_{zz}+2ff_{xx}+4ff_{yy}-f_z^2\\
    \hspace{11cm}-f_x^2-4f_y^2)\big{)}\Big{)}
  \end{array}
\right.
\end{equation}
\end{small}
If we denote the $3\times3$ identity matrix as $I_3$, the Riemannian metric $\tilde{g}$ in the coordinates $(x,y,z)$ can be expressed as $\frac{f}{z^2}I_3$. Therefore, we can observe that $X=\mbox{grad}\Phi$ if and only if the following system holds:
\begin{equation}\label{gradient eq for R^2 rtimes R^+}
 \left\{
  \begin{array}{l}
    \Phi_x=\frac{f\beta}{z} \\
    \Phi_y=\frac{f\gamma}{z}\\
    \Phi_z=\frac{f\alpha}{z}.
  \end{array}
\right.
\end{equation}
Using the above systems we construct flat shrinking, steady and expanding gradient and non-gradient Ricci solitons on the non-abelian Lie group $G=\Bbb{R}^2\rtimes\Bbb{R}^{+}$.
\begin{example}
In the coordinates $(x,y,z)$ let $f(x,y,z)=z^2$. The formula \eqref{Ricci curvature general} shows that, for $i,j=1\cdots 3$, $Ric(E_i,E_j)=0$, and so $(G,\tilde{g})$ is Ricci-flat. But we know that any three-dimensional Ricci-flat Riemannian manifold is flat (see \cite{Yau}). For any $\lambda\in\Bbb{R}$ suppose that $\alpha=\frac{-C_2y-C_4x+\lambda z+C_6}{z}$, $\beta=\frac{-C_1y+C_4z+\lambda x+C_5}{z}$ and $\gamma=\frac{C_1x+C_2z+\lambda y+C_3}{z}$, where $C_1,\cdots,C_6\in\Bbb{R}$. Now the system \eqref{Ricci Soliton eq for R^2 rtimes R^+} shows that $(G=\Bbb{R}^2\rtimes\Bbb{R}^{+},\tilde{g})$ is a non-abelian flat shrinking (if $\lambda>0$), steady (if $\lambda=0$) or expanding (if $\lambda<0$) Ricci soliton. \\
On the other hand, using the system \eqref{gradient eq for R^2 rtimes R^+}, we see that the Ricci soliton $(G,\tilde{g})$ is gradient if and only if $C_1=C_2=C_4=0$. If $C_1=C_2=C_4=0$ then the potential function of the gradient Ricci soliton $(G,\tilde{g})$ is $\Phi(x,y,z)=\frac{\lambda}{2}(x^2+y^2+z^2)+C_5x+C_3y+C_6z+C_7$, where $C_7\in\Bbb{R}$.
\end{example}

\subsection{\textbf{The Heisenberg group $H_3$}}\label{Heisenberg group}
In this subsection we study the Heisenberg group $H_3$ with the coordinates $(x,y,z)$ and the multiplication
$$(x_1,y_1,z_1).(x_2,y_2,z_2)=(x_1+x_2,y_1+y_2,z_1+z_2+\frac{1}{2}(x_2y_1-y_2x_1)).$$
Let $\tilde{g}$ be a $f$-left-invariant Riemannian metric such that the set $\{E_1:=\frac{\partial}{\partial x}-\frac{y}{2}\frac{\partial}{\partial z}, E_2:=\frac{\partial}{\partial y}+\frac{x}{2}\frac{\partial}{\partial z}, E_3:=\frac{\partial}{\partial z}\}$ is an orthogonal basis at any point and is orthonormal at $e=(0,0,0)$. So, with respect to this basis, the non-zero structural constants are $\alpha_{123}=-\alpha_{213}=1$. In this case the Levi-Civita connection of $\tilde{g}$ is as follows:

\begin{eqnarray*}\label{Levi-Civita Heisenberg group}
\resizebox{\textwidth}{!}{%
 \begin{tabular}{|c|c|c|c|}
  \hline
   $\tilde{\nabla}$ & $E_1$ & $E_2$ & $E_3$  \\
  \hline
  $E_1$ & $\frac{1}{2f}(f_1E_1-f_2E_2-f_3E_3)$ & $\frac{1}{2f}(f_2E_1+f_1E_2+fE_3)$ & $\frac{1}{2f}(f_3E_1-fE_2+f_1E_3)$ \\
  \hline
  $E_2$ & $\frac{1}{2f}(f_2E_1+f_1E_2-fE_3)$ & $\frac{1}{2f}(-f_1E_1+f_2E_2-f_3E_3)$ & $\frac{1}{2f}(fE_1+f_3E_2+f_2E_3)$ \\
  \hline
  $E_3$ & $\frac{1}{2f}(f_3E_1-fE_2+f_1E_3)$ & $\frac{1}{2f}(fE_1+f_3E_2+f_2E_3)$ & $\frac{1}{2f}(-f_1E_1-f_2E_2+f_3E_3)$ \\
  \hline
 \end{tabular}}
\end{eqnarray*}

Let $X=\alpha E_1+\beta E_2+\gamma E_3$ be vector field on $H_3$, where $\alpha, \beta, \gamma\in C^\infty(H_3)$. The relations \eqref{necessary and sufficient condition f-invariant} and \eqref{Ricci curvature general} show that the Riemannian manifold $(H_3,\tilde{g})$ with the vector field $X$ and expansion constant $\lambda$ is a Ricci soliton if and only if

\begin{equation}\label{Ricci Soliton eq for Heisenberg group}
 \left\{
  \begin{array}{l}
    Xf+2f\alpha_1=2\Big{(}\lambda f+\frac{1}{4f^2}\big{(}2f^2+4ff_{11}+2ff_{22}+2ff_{33}-4f_1^2-f_2^2-f_3^2\big{)}\Big{)}\\
    f(\beta_1+\alpha_2)=-\frac{1}{2f^2}(-ff_{21}-ff_{12}+3f_1f_2)\\
    f(\alpha_3+\beta+\gamma_1)=\frac{1}{2f^2}(ff_2-2ff_{13}+4ff_{31}-3f_1f_3)\\
    Xf+2f\beta_2=2\Big{(}\lambda f+\frac{1}{4f^2}\big{(}2f^2+2ff_{11}+4ff_{22}+2ff_{33}-f_1^2-4f_2^2-f_3^2\big{)}\Big{)}\\
    f(\gamma_2-\alpha+\beta_3)=-\frac{1}{2f^2}(ff_1+2ff_{23}-4ff_{32}+3f_2f_3)\\
    Xf+2f\gamma_3=2\Big{(}\lambda f-\frac{1}{4f^2}\big{(}2f^2-2ff_{11}-2ff_{22}-4ff_{33}+f_1^2+f_2^2+4f_3^2\big{)}\Big{)}.
  \end{array}
\right.
\end{equation}

In the coordinates $(x,y,z)$, the Riemannian metric $\tilde{g}$ is of the form
$$\left(
 \begin{array}{ccc}
 f(1+\frac{y^2}{4}) & -\frac{xyf}{4} & \frac{yf}{2} \\
 -\frac{xyf}{4} & f(1+\frac{x^2}{4}) & -\frac{xf}{2}\\
 \frac{yf}{2} & -\frac{xf}{2} & f \\
 \end{array}
\right),$$
so, we have $X=\mbox{grad}\Phi$ if and only if
\begin{equation}\label{gradient eq for Heisenberg group}
 \left\{
  \begin{array}{l}
    f=\frac{2\Phi_x-y\Phi_z}{2\alpha} \\
    f=\frac{2\Phi_y+x\Phi_z}{2\beta}\\
    f=\frac{-2y\Phi_x+2x\Phi_y+(x^2+y^2+4)\Phi_z}{2x\beta-2y\alpha+4\gamma}.
  \end{array}
\right.
\end{equation}
\begin{example}
For simplicity if we consider $f(x,y,z)=1$ then $\tilde{g}$ reduces to a left-invariant Riemannian metric on $H_3$. In this case, for a real constant $C$, if we put $\lambda=-\frac{3}{2}$, $\alpha=-x$, $\beta=-y$ and $\gamma=-2z+C$ (or equivalently if $X=-x\frac{\partial}{\partial x}-y\frac{\partial}{\partial y}+(-2z+C)\frac{\partial}{\partial z}$) then the above systems show that $(H_3,\tilde{g})$ is a non-gradient expanding Ricci soliton.
\end{example}


\bibliographystyle{amsplain}

\end{document}